\documentclass[12pt]{article}

\usepackage{amsmath}
\allowdisplaybreaks[4]

\usepackage[all]{xy}
\usepackage{amssymb}
\usepackage{amsthm}

\usepackage{bm}
\usepackage{indentfirst}

\usepackage{amscd}
\usepackage{enumitem}
\usepackage{verbatim}
\usepackage{eurosym}
\usepackage{float}
\usepackage{color}
\usepackage{dcolumn}
\usepackage[mathscr]{eucal}
\usepackage[all]{xy}
\usepackage{booktabs}
\usepackage{microtype}
\usepackage{hyperref}
\hypersetup{colorlinks=true,linkcolor=blue,citecolor=red}
\usepackage{enumitem}
\usepackage{mathrsfs}
\usepackage{amssymb}
\usepackage{amsfonts,ifpdf}
\usepackage{graphicx}
\usepackage{times}
\usepackage{float}
\usepackage{epstopdf}
\usepackage{cite}
\usepackage{youngtab}
\usepackage{ytableau}
\usepackage{array}
\ytableausetup
{mathmode, boxsize=0.9em}

\newtheorem{Theorem}{Theorem}[section]

\newtheorem{Lemma}[Theorem]{Lemma}

\newtheorem{pf}{proof}[section]
\theoremstyle{remark}

\newtheorem{remark}[Theorem]{Remark}

\makeatletter \@addtoreset{equation}{section} \makeatother
\makeindex 

\def\rk{\operatorname{rk}}

\def\A{\mathcal{A}}

\def\max{\mbox{\rm max}}

\def\k{\mbox{\rm $\bm{k}$}}

\allowdisplaybreaks[4]

\begin{document}

\begin{center}
{\Large\bf
Several New Generalizations of LYM Inequality
}\\ [7pt]
\end{center}

\begin{center}
	Zihao Huang$^{1}$,
	Weikang Liang$^{2}$,
	Yujiao Ma$^{*3}$
	and Suijie Wang$^{4}$\\[8pt]
	$^{1,2,3,4}$School of Mathematics\\
	Hunan University\\
	Changsha 410082, Hunan, P. R. China\\[12pt]
	
	Emails: $^{1}$zihaoh@hnu.edu.cn, $^{2}$kangkang@hnu.edu.cn, $^{3}$yujiaoma@hnu.edu.cn,  $^{4}$wangsuijie@hnu.edu.cn\\[15pt]

\end{center}

\vskip 3mm
\begin{abstract}
	The LYM inequality is a fundamental result concerning the sizes of subsets in a Sperner family. Subsequent studies on the LYM inequality have been generalized to families of $r$-decompositions, where all components are required to avoid chains of the same length. In this paper, we relax this constraint by allowing components of a family of $r$-decompositions to avoid chains of distinct lengths, and derive generalized LYM inequalities across all the relevant settings, including set-theoretic, $q$-analog, continuous analog, and arithmetic analog frameworks. Notably, the bound in our LYM inequalities does not depend on the maximal length of all forbidden chains.  Moreover, we extend our approach beyond $r$-decompositions to $r$-multichains, and establish analogous LYM inequalities.
	
	\vskip 6pt
	
	\\ [7pt]
	
\end{abstract}

\section{Background}
Let $B(n)$ be the Boolean lattice, which consists of all subsets of $[n] = \{1, \ldots, n\}$ and is ordered by inclusion. An \emph{antichain} in $B(n)$ is a family of subsets of $[n]$ where no two contain each other.
A fundamental question---what is the maximum size of an antichain in $B(n)$?---is answered by Sperner's Theorem:
\begin{itemize}
	\item[]
	\textbf{Sperner's Theorem} \cite{Sperner}. Any antichain $\mathcal{A}$ in $B(n)$ satisfies
	\begin{equation}
		|\mathcal{A}| \le \binom{n}{\lfloor n/2 \rfloor},	
	\end{equation}
	with equality if and only if $\mathcal{A}$ consists of all subsets of $[n]$ of size $\lfloor n/2 \rfloor$ or all subsets of size $\lceil n/2 \rceil$.
\end{itemize}
This result initiated a rich line of research on the Sperner property, which has since evolved into a well-developed theory within extremal set systems~\cite{Engel}.

The LYM inequality, which naturally implies Sperner's theorem as a special case, was independently obtained by Bollob\'{a}s\cite{Bollobas}, Lubell\cite{Lubell}, Yamamoto\cite{Yamamoto}, and Meshalkin\cite{Meshalkin}, and has since become a more general tool for analyzing antichains in $B(n)$.
\begin{itemize}
	\item[] \textbf{LYM inequality}\cite{Lubell, Meshalkin, Yamamoto, Bollobas}. Any antichain $\mathcal{A}$ in $B(n)$ satisfies the inequality
	\begin{equation}
		\sum_{A \in \mathcal{A}} \frac{1}{\binom{n}{|A|}} \le 1.
	\end{equation}
\end{itemize}

An \emph{$r$-decomposition} of $[n]$ is an $r$-tuple $(D_1, \dots, D_r)$ of disjoint subsets of $[n]$ such that $\bigcup_{i=1}^r D_i = [n]$.  
Given a family $\mathcal{D}$ of $r$-decompositions of $[n]$, we define 
\[
\mathcal{D}_k = \{ D_k \mid (D_1, \ldots,D_k,\ldots, D_r) \in \mathcal{D} \}
\]
for each $k \in \{1,\ldots,r\}$.
Building on Sperner's Theorem and the LYM inequality, Meshalkin\cite{Meshalkin} generalized Sperner's Theorem by extending the idea of antichain properties on single subsets to families of $r$-decompositions of $[n]$.

\begin{itemize}
	\item[] \textbf{Meshalkin's Theorem} \cite{Meshalkin}. 
	Let $\mathcal{D}$ be a family of $r$-decompositions of $[n]$ such that each $\mathcal{D}_k$ forms an antichain, then
	\begin{equation}\label{a1}
		|\mathcal{D}|  \leq  \binom{n}{\underbrace{ \lfloor n/r \rfloor, \dots, \lfloor n/r \rfloor}_{r - b}, \underbrace{ \lfloor n/r \rfloor + 1, \dots, \lfloor n/r \rfloor + 1}_{b}},
	\end{equation}
	where $n \equiv b \pmod{r}$.
\end{itemize}
Hochberg and Hirsch\cite{Hochberg} further extended the classic LYM inequality to a multinomial version, which is tailored to the setting of $r$-decompositions and directly implies Meshalkin's theorem. Their result is stated as follows:
\begin{itemize}
	\item[] \textbf{Hochberg-Hirsch Theorem}\cite{Hochberg}. Let $\mathcal{D}$ be a family of $r$-decompositions of $[n]$ such that each $\mathcal{D}_k$ forms an antichain. Then
		\begin{equation}\label{a2}
		\sum_{(D_1,\ldots,D_r) \in \mathcal{D}} \frac{1}{\binom{n}{|D_1|,\dots,|D_r|}} \le 1.
	\end{equation}
\end{itemize}

A \emph{chain of length~$t$} in $B(n)$ is a collection of $t+1$ subsets of $[n]$ where every pair of subsets is comparable. Recall that an antichain is precisely a family of subsets avoiding chains of length $1$. A family $\mathcal{A} \subseteq B(n)$ is said to be \emph{$t$-chain free} if $ \mathcal{A} $ contains no chain of length $t$. Erd\H{o}s \cite{Erdos},
Rota and Harper\cite{Rota and Harper} further extended the classical Sperner's Theorem and LYM inequality to the setting of $t$-chain free families, yielding the following results:
\begin{itemize}
	\item[] \textbf{Erd\H{o}s' Theorem}\cite{Erdos}. 
	If $\mathcal{A} \subseteq B(n)$ is $t$-chain free, then
	\begin{equation}
		|\mathcal{A}| \leq \binom{n}{\left \lfloor \frac{n+1}{2} \right \rfloor } + \cdots + \binom{n}{\left \lfloor \frac{n+t}{2} \right \rfloor }.
	\end{equation}

	\item[] \textbf{Rota-Harper Theorem}\cite{Rota and Harper}. 
	If $\mathcal{A} \subseteq B(n)$ is $t$-chain free, then
	\begin{equation}\label{b}
		\sum_{A\in \mathcal{A} }^{} \frac{1}{\binom{n}{\left | A \right |}}\leq t.
	\end{equation}
\end{itemize}

Beck and Zaslavsky\cite{Beck and Zaslavsky1} generalized these results further: they considered the families of $r$-decompositions with each $\mathcal{D}_k$ is $t$-chain free, and thus extended these results to a more general setting.
\begin{itemize}
	\item[] \textbf{Beck-Zaslavsky Theorem} \cite{Beck and Zaslavsky1}.  
	Let $\mathcal{D}$ be a family of $r$-decompositions of $[n]$ such that each $\mathcal{D}_k$ is $t$-chain free. Then
	\begin{equation}\label{a3}
		\sum_{(D_1, \ldots, D_r) \in \mathcal{D}} \frac{1}{\binom{n}{|D_1|, \ldots, |D_r|}} \leq t^{r-1}.	
	\end{equation}
	Consequently, $|\mathcal{D}| \leq m_1 + \cdots + m_{t^{r-1}}$ where $m_1, \ldots, m_{t^{r-1}}$ are the $t^{r-1}$ largest multinomial coefficients $\binom{n}{a_1, \ldots, a_r}$ with $a_1 + \cdots + a_r = n$.
\end{itemize}

Up to this point, we have focused on results in the Boolean lattice $B(n)$. The corresponding $q$-analogue is given by the subspace lattice $L_q(n)$ of $\mathbb{F}_q^n$ over a finite field $\mathbb{F}_q$. Paralleling the definition of $t$-chain free families in $B(n)$, a family $\mathcal{A}\subseteq L_q(n)$ is called \emph{$t$-chain free} if $\mathcal{A}$ contains no chain of length $t$ in $L_q(n)$. The number of $k$-dimensional subspaces of $\mathbb{F}_q^n$ is denoted by $\genfrac{[}{]}{0pt}{}{n}{k}_{q}$, called $q$-Gaussian coefficient. In the setting of the subspace lattice $L_q(n)$, Rota and Harper  \cite{Rota and Harper} established the $q$-analog of the LYM inequality as follows:
\begin{itemize}
	\item[] \textbf{Rota-Harper Theorem} \cite{Rota and Harper}.  
	Any $t$-chain free family $\mathcal{A} \subseteq L_q(n)$ satisfies
	\begin{equation}\label{a4}
		\sum_{A \in \mathcal{A}} \frac{1}{\genfrac{[}{]}{0pt}{}{n}{\dim(A)}_{q}} \le t.
	\end{equation}
\end{itemize}
Similarly, the concept of $r$-decompositions of $[n]$ can be extended to the setting of $L_q(n)$.
An \emph{$r$-decomposition} (also called an \emph{$r$-Meshalkin sequence}) of $\mathbb{F}_q^n$ is a tuple $(D_1, \ldots, D_r)$ where each $D_i \in L_q(n)$ and $D_1 \oplus \cdots \oplus D_r =\mathbb{F}_q^n$.
Given a family $\mathcal{D}$ of $r$-decompositions of $\mathbb{F}_q^n$, define
\[
\mathcal{D}_k := \left\{ D_k \mid (D_1, \ldots,D_k,\ldots, D_r) \in \mathcal{D} \right\}
\]
for each $k \in \{1, \ldots, r\}$.
Beck and Zaslavsky\cite{Beck and Zaslavsky2} further proved the $q$-analog of Meshalkin's theorem, stated as follows:
\begin{itemize}
	\item[] \textbf{Beck–Zaslavsky Theorem} \cite{Beck and Zaslavsky2}.  
	Let $\mathcal{D}$ be a family of $r$-decompositions of $\mathbb{F}_q^n$ such that each $\mathcal{D}_k$ is $t$-chain free. Then
	\begin{equation}\label{a5}
		\sum_{(D_1, \ldots, D_r) \in \mathcal{D}} \frac{1}{\genfrac{[}{]}{0pt}{}{n}{\dim(D_1),\ldots,\dim(D_r)}_{q} \prod_{i<j} q^{\dim(D_i)\dim(D_j)}} \le t^{r-1}.
	\end{equation}
	Consequently, $|\mathcal{D}| \leq m_1 + \cdots + m_{t^{r-1}}$ where $m_1, \ldots, m_{t^{r-1}}$ are the $t^{r-1}$ largest terms of $
	\genfrac{[}{]}{0pt}{}{n}{a_1,\ldots,a_r}_{q} \prod_{i<j} q^{a_i a_j}$ with $a_1 + \cdots + a_r = n$.
\end{itemize}

The continuous analogs of the LYM inequality and Meshalkin's theorem were first introduced by Klain and Rota~\cite{Klain and Rota, KlainRota}.  
Denote by $\mathrm{Mod}(n)$ the subspace lattice of $\mathbb{R}^n$. A family $\mathcal{A} \subseteq \mathrm{Mod}(n)$ is called \emph{$t$-chain free} if $\mathcal{A}$ contains no chain of length $t$ in $\mathrm{Mod}(n)$.
Let $\mathrm{Gr}(n, k)$ denote the Grassmannian of all $k$-dimensional linear subspaces of $\mathbb{R}^n$.  
Notably, $\mathrm{Gr}(n, k)$ forms a compact smooth manifold equipped with an orthogonal invariant measure $\nu_k^n$. The total measure of $\mathrm{Gr}(n,k)$ is denoted by $\genfrac{[}{]}{0pt}{}{n}{k}_{\mathbb{R}}$.
Klain and Rota established the following continuous analog of the LYM inequality:
\begin{itemize}
	\item[] \textbf{Klain–Rota Theorem} \cite{Klain and Rota}.
	For any \( t \)-chain free family \(\mathcal{A}\) of subspaces of \(\mathbb{R}^n\), we have
	\begin{equation}\label{a6}
		\sum_{k=0}^{n} \frac{\nu_{k}^{n}\big(\mathcal{A} \cap \mathrm{Gr}(n,k)\big)}{\genfrac{[}{]}{0pt}{}{n}{k}_{\mathbb{R}}} \le t.
	\end{equation}
\end{itemize} 
An \emph{$r$-decomposition} of $\mathbb{R}^n$ is a tuple $(D_1, \ldots, D_r)$ of nonzero subspaces such that $\mathbb{R}^n = \bigoplus_{i=1}^r D_i$, with $D_i \perp D_j$ for all $i \ne j$.  
Let $\mathcal{D}(\mathbb{R}^n, r)$ denote the set of all $r$-decompositions of $\mathbb{R}^n$.
Given positive integers $a_1, \ldots, a_r$ with $a_1 + \cdots + a_r = n$ and $\mathcal{D} \subseteq \mathcal{D}(\mathbb{R}^n, r)$, define 
$\mathcal{D}_{a_1, \ldots, a_r}$ to be the subset of all $r$-decompositions $(D_1, \ldots, D_r) \in \mathcal{D}$ such that $\dim(D_i) = a_i$ for all $i = 1, \ldots, r$.  
Clearly, we have the disjoint union
\[
\mathcal{D}(\mathbb{R}^n, r) = \bigsqcup_{a_1 + \cdots + a_r = n} 
\mathcal{D}(\mathbb{R}^n, r)_{a_1, \ldots, a_r}.
\]
Each $\mathcal{D}(\mathbb{R}^n, r)_{a_1, \ldots, a_r}$ is a compact smooth manifold equipped with a measure $\nu_{a_1, \ldots, a_r}^n$, which is induced naturally from the measure $\nu_k^n$ on $\mathrm{Gr}(n,k)$.
We denote by $\genfrac{[}{]}{0pt}{}{n}{a_1,\ldots,a_r}_{\mathbb{R}}$ the total measure of $\mathcal{D}(\mathbb{R}^n, r)_{a_1, \ldots, a_r}$ with respect to $\nu_{a_1, \ldots, a_r}^n$.
The measure $\nu_{n;r}$ on $\mathcal{D}(\mathbb{R}^n, r)$ is defined for any $\mathcal{D} \subseteq \mathcal{D}(\mathbb{R}^n, r)$ by
\[
\nu_{n;r}(\mathcal{D}) = \sum_{a_1 + \cdots + a_r = n} \nu_{a_1, \ldots, a_r}^n(\mathcal{D}_{a_1, \ldots, a_r}).
\]
Let $\mathcal{D} \subseteq \mathcal{D}(\mathbb{R}^n, r)$, we define
\[
\mathcal{D}_k := \{D_k \mid (D_1, \ldots,D_k,\ldots, D_r) \in \mathcal{D} \}
\]
for each $k \in \{1,\ldots,r\}$. Klain and Rota also established the following continuous analog of Meshalkin's theorem:
\begin{itemize}
	\item[] \textbf{Klain–Rota Theorem} \cite{Klain and Rota}.
	Let $\mathcal{D} \subseteq \mathcal{D}(\mathbb{R}^n,r)$ be a family such that $\mathcal{D}_k$ is an antichain for each $k \in \{1, \ldots, r\}$. Then
	\begin{equation}\label{a7}
		\sum_{a_1 + \cdots + a_r = n}\dfrac{\nu_{a_1,\ldots,a_r}^n(\mathcal{D}_{a_1,\ldots,a_r})}{\genfrac{[}{]}{0pt}{}{n}{a_1,\ldots,a_r}_{\mathbb{R}}} \le 1.
	\end{equation}
	Consequently, if $n \equiv b \pmod{r}$, then
\[
\nu_{n;r}(\mathcal{D}) \leq {\genfrac{[}{]}{0pt}{}{n}{\underbrace{\lfloor n/r \rfloor, \dots, \lfloor n/r \rfloor}_{r - b},\underbrace{\lfloor n/r \rfloor + 1, \dots, \lfloor n/r \rfloor + 1}_{b}}}_{\raisebox{5mm}{$\mathbb{R}$}}
\]
\end{itemize}

Beyond the $q$-analogs and continuous analogs, one can also consider the arithmetic analogs of the LYM inequality.
Let $n$ be a positive integer, and let $\mathrm{Div}(n)$ denote the lattice of all positive divisors of $n$, ordered by divisibility.  
For a divisor $x \in \mathrm{Div}(n)$, its rank is the total number of prime factors of $x$, counting multiplicities. Let $W_i(n)$ be the number of elements in $\mathrm{Div}(n)$ of rank $i$. Anderson\cite{Ian} proved the following arithmetic analog of the LYM inequality:
\begin{itemize}
	\item[] \textbf{Anderson's Theorem} \cite{Ian}.
	For any antichain $\mathcal{A} \subseteq \mathrm{Div}(n)$, we have
	\begin{equation}\label{a8}
		\sum_{i \ge 0} \frac{|\mathcal{A}_i|}{W_i(n)} \le 1,
	\end{equation}
	where $\mathcal{A}_i$ denotes the elements of $\mathcal{A}$ of rank $i$.
\end{itemize}

\section{Main results}
In all known generalizations of the LYM inequality and its various analogs, the results have been established only for families $\mathcal{D}$ of $r$-decompositions in which every component avoids chains of the same fixed length, i.e., all $\mathcal{D}_k$ are $t$-chain free for a single integer $t$. 

One of the central contributions of this paper is to lift this restriction: we extend these results to the more general and flexible setting where different components of the family $\mathcal{D}$ may avoid chains of 
distinct lengths. Specifically, we allow its $k$-th component family $\mathcal{D}_k$ to be $t_k$-chain free, where $t_1, \ldots, t_r$ are arbitrary given positive integers. We establish such extensions across four fundamental lattice structures: the Boolean lattice $B(n)$, the subspace lattice $L_q(n)$ of $\mathbb{F}_q^n$, the subspace lattice $\mathrm{Mod}(n)$ of $\mathbb{R}^n$, and the divisor lattice $\mathrm{Div}(n)$. We derive two types of results: the corresponding generalized LYM inequalities, and the corresponding upper bounds on the size or measure of the $r$-decomposition family $\mathcal{D}$. These upper bounds rely on a key parameter $\sigma = \frac{t_1 t_2 \cdots t_r}{\mathrm{max}\{t_1, \ldots, t_r\}}$, with the bound taking the form of the sum of the $\sigma$ largest relevant summands. A summary of all these main results is provided in Table \ref{g0}.

\begin{remark}
	It is noted that the value of $\sigma = \frac{t_1 t_2 \cdots t_r}{\mathrm{max}\{t_1, \ldots, t_r\}}$ is independent of the specific value of $\max\{t_1, \ldots, t_r\}$. In other words, the component associated with the longest forbidden chain plays no essential role in the derived LYM inequalities and the corresponding bounds, which remain unchanged no matter how large $\max\{t_1, \ldots, t_r\}$ is.  For $r=2$, this phenomenon admits a straightforward interpretation: by the definition of the $2$-decomposition, if $\mathcal{D}_1$ is $t_1$-chain free, then $\mathcal{D}_2$ is also $t_1$-chain free; consequently, $\mathcal{D}_2$ is trivially $t_2$-chain free when $t_1 \leq t_2$. When $r\geq 3$, however, a clear interpretation of this phenomenon remains elusive.
\end{remark}

\begin{table}[h]
\centering
\renewcommand{\arraystretch}{2.15}
\setlength{\tabcolsep}{11pt}
\begin{tabular}{|c|c|c|}
	\hline
	\textbf{Lattice} & \textbf{LYM Inequality} & \textbf{Upper bound}\\
	\hline
	$B(n)$ & 
	$\displaystyle
	\sum_{a_{1}+\cdots+a_{r}=n}^{} \frac{|\mathcal{D}_{a_{1},\dots,a_{r}}|}{\binom{n}{{a_1, \ldots, a_r}} } \leq \sigma$ & $|\mathcal{D}| \leq  \sum\binom{n}{a_1, \ldots, a_r}$\\[10pt]
	
	\hline
	$L_q(n)$& 
	$\displaystyle 
	\;\sum_{a_1 + \cdots + a_r = n}\frac{|\mathcal{D}_{a_{1},\dots,a_{r}}|}{\genfrac{[}{]}{0pt}{}{n}{a_1,\ldots,a_r}_{q}\prod\limits_{i<j}q^{a_ia_j}}\le \sigma\;\;
	$ & \;\;$|\mathcal{D}| \leq \sum \genfrac{[}{]}{0pt}{}{n}{a_1,\ldots,a_r}_{q}\prod\limits_{i<j}q^{a_ia_j}$\;\\[15pt]
	\hline
	$\mathrm{Mod}(n)$ & 
	$\displaystyle 
	\sum_{a_{1}+\cdots+a_{r}=n}\dfrac{\nu_{a_1,\ldots,a_r}^n(\mathcal{D}_{a_1,\ldots,a_r})}{\genfrac{[}{]}{0pt}{}{n}{a_1,\ldots,a_r}_{\mathbb{R}}}\le \sigma$ & $\nu_{n;r}(\mathcal{D}) \leq \sum\genfrac{[}{]}{0pt}{}{n}{a_1,\ldots,a_r}_{\mathbb{R}}$\\[11pt]
	
	\hline
	$\mathrm{Div}(n)$ & 
	$\displaystyle 
	\sum_{a_{1}+\cdots+a_{r}=\rk(n)}^{} \frac{|\mathcal{D}_{a_{1},\dots,a_{r}}|}{N_{a_1, \ldots,a_r}(n)} \leq \sigma$ & $ |\mathcal{D}| \leq \sum N_{a_1, \ldots,a_r}(n)$\\[10pt]
	\hline
\end{tabular}
		\caption{Generalizations of LYM inequality in classical lattices}
		\label{g0}
\end{table}

In addition to extending results for $r$-decompositions, we further generalize our framework to a closely related structure: 
$r$-multichains. Formally, an $r$-multichain of a lattice is a tuple $(C_1, \ldots, C_r)$ of elements of the lattice such that $C_1 \leq C_2 \leq \cdots \leq C_r$. For a family $\mathcal{C}$ of $r$-multichains, we define $
\mathcal{C}_k = \{C_k \mid (C_1, \ldots, C_k, \ldots, C_r) \in \mathcal{C}\}$
for each $k\in \{1, \ldots, r\}$, and
\[
\mathcal{C}_{a_1,\dots,a_r} = \{(C_1, \ldots, C_r) \in \mathcal{C} \mid \rk(C_i) = a_1 + \cdots + a_i \text{ for all } i = 1, \ldots, r\},
\]
where $\rk(\cdot)$ denotes the rank function of this lattice, and $a_1,\ldots,a_r$ are non-negative integers.

Following the same research paradigm as our $r$-decomposition results, another key central contribution of this paper is to establish two types of results for $r$-multichains $\mathcal{C}$ such that each $\mathcal{C}_k$ is $t_k$-chain free for arbitrary positive integers $t_1,\ldots,t_r$: the generalized LYM inequalities, and the corresponding upper bounds on the size or measure of $\mathcal{C}$. A summary of these results is provided in Table \ref{tab:g2}, where $\tau$ denotes the product $t_1 t_2 \cdots t_r$. In the last column, each summation is taken over the $\tau$ largest terms among all summands.
\begin{table}[H]
	\centering
	\renewcommand{\arraystretch}{2.15}
	\setlength{\tabcolsep}{12.3pt}
	\begin{tabular}{|c|c|c|}
		\hline
		\textbf{Lattices} & \textbf{LYM Inequality} & \textbf{Upper bound}\\
		\hline
		
		\textbf{$B(n)$} & 
		$\displaystyle 
		\sum_{a_1 + \cdots + a_{r+1} = n } \frac{|\mathcal{C}_{a_{1},\dots,a_{r}}|}{\binom{n}{{a_1,\ldots,a_{r+1}}} }\leq \tau
		$ & $|\mathcal{C}| \leq  \sum\binom{n}{a_1, \ldots, a_{r+1}}$\\[10pt]
		
		\hline
		\textbf{$L_q(n)$} & 
		$\displaystyle 
		\sum_{a_1 + \cdots + a_{r+1} = n}\frac{|\mathcal{C}_{a_{1},\dots,a_{r}}|}{\genfrac{[}{]}{0pt}{}{n}{a_1,\ldots,a_{r+1}}_{q}}\le \tau
		$ & $|\mathcal{C}| \leq \sum \genfrac{[}{]}{0pt}{}{n}{a_1,\ldots,a_{r+1}}_{q}$\\[10pt]
		
		\hline
		\textbf{$\mathrm{Mod}(n)$} & 
		$\;\;\displaystyle 
		\sum_{a_{1}+\cdots+a_{r+1}=n}\dfrac{\mu_{a_1,\ldots,a_r}^n(\mathcal{C}_{a_1,\ldots,a_{r}})}{\genfrac{[}{]}{0pt}{}{n}{a_1,\ldots,a_{r+1}}_{\mathbb{R}}}\le \tau
		$ \;\;&\;$\mu_{n;r}(\mathcal{C}) \leq \sum\genfrac{[}{]}{0pt}{}{n}{a_1,\ldots,a_{r+1}}_{\mathbb{R}}$\;\\[10pt]
		
		\hline
		\textbf{$\mathrm{Div}(n)$} & 
		$\displaystyle 
		\sum_{a_{1}+\cdots+a_{r+1}=\rk(n)}^{} \frac{|\mathcal{C}_{a_{1},\dots,a_{r}}|}{N_{a_1,\ldots,a_{r+1}}(n)} \leq \tau$ & $ |\mathcal{C}| \leq \sum N_{a_1, \ldots,a_{r+1}}(n)$ \\[10pt]
		
		\hline
	\end{tabular}
	\caption{LYM inequalities for $r$-multichains in classical lattices}
	\label{tab:g2}
\end{table}

\section{Generalizations of the set-theoretic LYM inequality}\label{sec2}

Recall that an \emph{$r$-decomposition} of $[n]$ is a tuple $\left ( D_{1},\dots ,D_{r} \right )$ of subsets of $[n]$ where $\bigcup_{i=1}^r D_i = [n]$ and $D_i \cap D_j = \emptyset$ for all $i \neq j$. Given a family $\mathcal{D}$ of $r$-decompositions of $[n]$ and non-negative integers $a_1,\ldots,a_r$ with $a_1 + \cdots + a_r =n$, we write 
\[
\mathcal{D}_k = \{D_k \mid (D_1, \ldots, D_k, \ldots, D_r) \in \mathcal{D} \}\]
for each $k \in \{1, \ldots, r\}$, and 
\[\mathcal{D}_{a_{1},\dots,a_{r}}=\{(D_{1},\dots ,D_{r}) \in \mathcal{D} \mid |D_i| = a_i,\; i=1,\ldots,r\}.\]

\begin{Theorem}\label{Dec}
	Let $t_1, \ldots, t_r$ be positive integers and $\sigma = \frac{t_1t_2\cdots t_r}{\mathrm{max}\{t_1, \ldots, t_r\}}$. Suppose $\mathcal{D}$ is a family of $r$-decompositions of $[n]$ such that $\mathcal{D}_k$ is $t_k$-chain free for each $k \in \{1, \ldots, r\}$. Then
	\begin{equation}\label{1.1}
		\sum_{a_{1}+\cdots+a_{r}=n}^{} \frac{|\mathcal{D}_{a_{1},\dots,a_{r}}|}{\binom{n}{{a_1, \ldots, a_r}} } \leq \sigma.
	\end{equation}
	Consequently, $$|\mathcal{D}| \leq m_1+\cdots+m_\sigma,$$
	where $m_1,\ldots,m_\sigma$ are the $\sigma$ largest $\binom{n}{{a_1, \ldots, a_r}}$ for non-negative integers $a_1, \ldots, a_r$ with $a_1+\cdots+a_r = n$.
\end{Theorem}

Before moving to the proof of Theorem \ref{Dec}, we state a crucial lemma that plays a key role in our proof.


\begin{Lemma}\label{d}
	Suppose that $c_{1} \ge \cdots \ge c_{n} > 0$. If $c_{i}\ge x_{i}\ge 0$  for $i \in \{1, \ldots, n\}$, and if
	\begin{equation}\label{lem}
		x_{1} + \cdots + x_{n} > c_{1}+ \cdots +c_{t},
	\end{equation}
	then
	\begin{equation*}
		\sum_{k=1}^{n}\frac{x_{k}}{c_{k}} > t.
	\end{equation*}
\end{Lemma}
\begin{proof}
	Let $y_i=c_i-x_i$ for each $i\in \{1, \ldots, n\}$. Then the inequality \eqref{lem} reduces to
	\begin{equation*}
		x_{t+1}+\cdots+x_n>y_1+\cdots+y_{t}.
	\end{equation*}
	Therefore, we have
	\begin{equation*}
		\begin{aligned}
			\sum_{k=1}^{n}\frac{x_{k}}{c_{k}}
			&=t-\sum_{k=1}^{t}\frac{y_k}{c_{k}}+\sum_{k=t+1}^{n}\frac{x_{k}}{c_{k}}\\
			&\geq t+\sum_{k=t+1}^{n}\frac{x_{k}}{c_{k}}-\frac{y_1+\cdots+y_{t}}{c_{t+1}}\\
			&> t+\sum_{k=t+1}^{n}\frac{x_{k}}{c_{k}}-\frac{x_{t+1}+\cdots+x_n}{c_{t+1}}\\
			&\geq t.
		\end{aligned}
	\end{equation*}
	
\end{proof}

\begin{proof}[Proof of Theorem \ref{Dec}]
 	Without loss of generality, assume that $t_r$ is the maximal integer among $\{t_1, \cdots, t_r\}$. We proceed by induction on $r$ to prove the inequality \eqref{1.1}. For $r =2$, note that any $2$-decomposition $(D_1, D_2)$ of $[n]$ satisfies $D_2 = [n]\setminus D_1$. Then, if $\mathcal{D}_1$ is $t_1$-chain free, it follows that $\mathcal{D}_2$ is also $t_1$-chain free, and hence $t_2$-chain free. Therefore, the inequality \eqref{1.1} is equivalent to \eqref{b} in Rota-Harper Theorem by choosing the collection $\A$ to be $\mathcal{D}_1$.
	Now suppose that $r>2$ and that the inequality \eqref{1.1} holds for $r-1$. Then
	\begin{equation}\label{c1}
		\begin{aligned}
			\sum_{a_{1}+\cdots+a_{r}=n}^{} \frac{|\mathcal{D}_{a_{1},\dots,a_{r}}|}{\binom{n}{{a_1,\ldots,a_r}} }&=\sum_{(D_1,\ldots, D_r)\in \mathcal{D} }^{} \frac{1}{\binom{n}{{|D_1|,\ldots,|D_r|}} }\\
			&=\sum_{(D_1,\ldots, D_r)\in \mathcal{D} }^{} \frac{1}{\binom{n}{|D_1|}}\frac{1}{\binom{n-|D_1|}{{|D_2|,\ldots,|D_r|}} } \\
			&=\sum_{A\in \mathcal{D}_1}\frac{1}{\binom{n}{|A|}}\sum\limits_{\substack{{ (D_1,\ldots, D_r)\in \mathcal{D}}\\{D_{1} = A}}}\frac{1}{\binom{n-|A|}{{|D_2|,\ldots,|D_r|}}}.
		\end{aligned}
	\end{equation}
	It follows from the induction hypothesis that
	\[
	\sum\limits_{\substack{{ (D_1,\ldots, D_r)\in \mathcal{D}}\\{D_{1} = A}}}\frac{1}{\binom{n-|A|}{{|D_2|,\ldots,|D_r|}}} \leq t_2\cdots t_{r-1}  .
	\]
	Therefore, the last expression in \eqref{c1} is at most $\sigma = t_1t_2\cdots t_{r-1}$ by \eqref{b}, which completes the proof of \eqref{1.1}.
	
	The number of multinomial coefficients $\binom{n}{a_{1},\ldots,a_{r}}$ that satisfy $a_1+\cdots + a_r=n$ is $\binom{n+r-1}{r-1}$. Note that
	\[
	\mathcal{D} = \bigsqcup_{a_{1}+\cdots+a_{r}=n}\mathcal{D}_{a_{1},\dots,a_{r}}
	\]
	and $0\le|\mathcal{D}_{a_{1},\dots,a_{r}}|\le \binom{n}{{a_1,\ldots,a_r}}$ for each $\mathcal{D}_{a_{1},\dots,a_{r}}$.  When $\sigma > \binom{n+r-1}{r-1}$, we assume that $m_i = 0$ for all $i > \binom{n+r-1}{r-1}$ and then $|\mathcal{D}| \leq m_1+\cdots+m_\sigma$ trivially holds. For the case $\sigma \leq \binom{n+r-1}{r-1}$, suppose for contradiction that $|\mathcal{D}| > m_1+\cdots+m_\sigma$.
	Then by Lemma \ref{d} we have
	\begin{equation*}
	  	\sum_{a_{1}+\cdots+a_{r}=n}^{} \frac{|\mathcal{D}_{a_{1},\dots,a_{r}}|}{\binom{n}{{a_1,\ldots,a_r}} } > \sigma,
	\end{equation*}
	which contradicts the inequality \eqref{1.1}. This completes the proof.
\end{proof}

\begin{remark}
	The upper bound given in Theorem \ref{Dec} is not tight in general. Consider the case where $n=4$ and $r = 3$. When $t_1 = 1$ and $t_2 = t_3 = 2$, Theorem \ref{Dec} provides the upper bound $$\binom{4}{1,1,2} + \binom{4}{1,2,1} = 24$$ on the size of $\mathcal{D}$. This bound can be attained by the following collection:
\begin{equation*}
	 \left\{
	\begin{array}{llll}
		(\{1\}, \{2\}, \{3,4\}), & (\{2\}, \{1\}, \{3,4\}), & (\{3\}, \{1\}, \{2,4\}), & (\{4\}, \{1\}, \{2,3\}), \\
		(\{1\}, \{3\}, \{2,4\}), & (\{2\}, \{3\}, \{1,4\}), & (\{3\}, \{2\}, \{1,4\}), & (\{4\}, \{2\}, \{1,3\}), \\
		(\{1\}, \{4\}, \{2,3\}), & (\{2\}, \{4\}, \{1,3\}), & (\{3\}, \{4\}, \{1,2\}), & (\{4\}, \{3\}, \{1,2\}), \\
		(\{1\}, \{3,4\}, \{2\}), & (\{2\}, \{3,4\}, \{1\}), & (\{3\}, \{2,4\}, \{1\}), & (\{4\}, \{2,3\}, \{1\}), \\
		(\{1\}, \{2,4\}, \{3\}), & (\{2\}, \{1,4\}, \{3\}), & (\{3\}, \{1,4\}, \{2\}), & (\{4\}, \{1,3\}, \{2\}), \\
		(\{1\}, \{2,3\}, \{4\}), & (\{2\}, \{1,3\}, \{4\}), & (\{3\}, \{1,2\}, \{4\}), & (\{4\}, \{1,2\}, \{3\})
	\end{array}
	\right\}.
\end{equation*}
	However, when $t_1 = 1$ and $t_2 = t_3 = 3$, a computational enumeration shows that the maximum size of $\mathcal{D}$ is 28, which is strictly smaller than the upper bound $$\binom{4}{2,1,1} + \binom{4}{1,2,1} + \binom{4}{1,1,2} = 36$$ given by Theorem \ref{Dec}. 
\end{remark}

Next, we turn to establishing the LYM inequality for $r$-multichains.
An $r$-multichain in $B(n)$ is an $r$-tuple $(C_1, \ldots, C_r)$ of elements of $B(n)$ such that $C_1 \subseteq C_2 \subseteq \cdots \subseteq C_r$.
Let $\mathcal{C}$ be a family of $r$-multichains, and let $a_1, \ldots, a_r$ be non-negative integers with $a_1 + \cdots + a_r \leq n$. Define $\mathcal{C}_k = \{C_k \mid (C_1, \ldots, C_k, \ldots, C_r) \in \mathcal{C}\}$ for each $k \in \{1, \ldots, r\}$, and
\[
\mathcal{C}_{a_1,\dots,a_r} = \{(C_1, \ldots, C_r) \in \mathcal{C} \mid |C_i| = a_1 + \cdots + a_i \text{ for all } i = 1, \ldots, r\}.
\]
\begin{Theorem}\label{Ch}
	Let $t_1, \ldots, t_r$ be positive integers and $\tau = t_1t_2\cdots t_r$. Suppose $\mathcal{C}$ is a family of $r$-multichains in $B(n)$ such that $\mathcal{C}_k$ is $t_k$-chain free for each $k \in \{1, \ldots, r\}$. Then
	\begin{equation}\label{1.2}
		\sum_{a_1 + \cdots + a_{r+1} = n } \frac{|\mathcal{C}_{a_{1},\dots,a_{r}}|}{\binom{n}{{a_1,\ldots,a_{r+1}}} }\leq \tau.
	\end{equation}
	Consequently,
	\[
	|\mathcal{C}| \leq m_1 + \cdots + m_\tau,
	\]
	where $m_1,\ldots,m_\tau$ are the $\tau$ largest $\binom{n}{{a_1, \ldots, a_{r+1}}}$ for non-negative integers $a_1, \ldots, a_r$ with $a_1+\cdots+a_{r+1} = n$.
\end{Theorem}

\begin{remark}\label{remk2}
	Denote by $\mathcal{D}([n], r)$ and $\mathcal{C}([n],r)$ the set of all $r$-decompositions and $r$-multichains of $[n]$, respectively.
	It can be easy to verify that the map $\phi : \mathcal{D}([n], r) \to \mathcal{C}([n],r-1)$, defined by
	$$
	\phi(D_1, \ldots, D_r) = (D_1, D_1 \cup D_2, \ldots, D_{1} \cup \cdots \cup D_{r-1}),
	$$
	is a bijection. However, the image $\phi(\mathcal{D})$ of the family $\mathcal{D}$ appearing in Theorem \ref{Dec} does not necessarily satisfy the condition that each $\phi(\mathcal{D})_k$ is also $t_k$-chain-free. A concrete counterexample is given as follows: Let $n = 4$ and $r= 3$. Consider the collection $\mathcal{D} = \{(\{1\},\{3\},\{2,4\}),(\{2,3\},\{1\},\{4\})\}$. Then $\phi(\mathcal{D}) = \{(\{1\},\{1,3\}),(\{2,3\},\{1,2,3\})\}$. We observe that $\mathcal{D}_2$ is $1$-chain free, while $\phi(\mathcal{D})_2$ fails to preserve the $1$-chain free property. Therefore, Theorem \ref{Ch} cannot be deduced from Theorem \ref{Dec} via this bijection. 
\end{remark}

\begin{proof}[Proof of theorem \ref{Ch}]
	 We proceed by induction on $r$ to prove the inequality \eqref{1.2}. For $r=1$, \eqref{1.2} is actually the generalized LYM inequality \eqref{b}. Suppose that $r>1$ and that the inequality \eqref{1.2} holds for $r-1$. Then
	\begin{equation*}
		\begin{aligned}
			\sum_{a_1 + \cdots + a_{r+1} = n } \frac{|\mathcal{C}_{a_{1},\dots,a_{r}}|}{\binom{n}{{a_1,\ldots,a_{r+1}}} }
			&=\sum_{(C_1,  \ldots, C_r) \in \mathcal{C} } \frac{1}{\binom{n}{{|C_1|,|C_2-C_1|,\ldots,|C_r-C_{r-1}|,|[n]-C_r|}}}\\
			&=\sum_{(C_1,  \ldots, C_r)\in \mathcal{C} } \frac{1}{\binom{n}{|C_1|}}\frac{1}{\binom{n-|C_1|}{{|C_2-C_1|,\ldots,|C_r-C_{r-1}|,|[n]-C_r|}} } \\
			&=\sum_{{A\in \mathcal{C}_1}}^{}\frac{1}{\binom{n}{|A|}}\sum\limits_{\substack{{ (C_1,  \ldots, C_r)\in \mathcal{C}}\\{C_{1} = A}}}\frac{1}{\binom{n-|A|}{{|C_2-C_1|,\ldots,|C_r-C_{r-1}|,|[n]-C_r|}}}.
		\end{aligned}
	\end{equation*}
	It follows from the induction hypothesis and \eqref{b} that the last expression is at most $\tau=t_1t_2\cdots t_r$.
	
The number of multinomial coefficients $\binom{n}{a_{1},\ldots,a_{r+1}}$ that satisfy $a_1+\cdots + a_{r+1}=n$ is $\binom{n+r}{r}$. Note that
\[
\mathcal{C} = \bigsqcup_{a_{1}+\cdots+a_{r+1}=n}\mathcal{C}_{a_{1},\dots,a_{r}}
\]
and $0\le|\mathcal{C}_{a_{1},\dots,a_{r}}|\le \binom{n}{{a_1,\ldots,a_{r+1}}}$ for each $\mathcal{C}_{a_{1},\dots,a_{r}}$. The inequality $|\mathcal{C}| \leq m_1 + \cdots + m_\tau$ trivially holds when $\tau > \binom{n+r}{r}$, assuming that $m_i = 0$ for all $i > \binom{n+r}{r}$. For the case $\tau \leq \binom{n+r}{r}$, suppose for contradiction that $|\mathcal{C}| > m_1 + \cdots + m_\tau$.
Then, by Lemma \ref{d} we have
\begin{equation*}
	\sum_{a_{1}+\cdots+a_{r+1}=n} \frac{|\mathcal{C}_{a_{1},\dots,a_{r}}|}{\binom{n}{{a_1,\ldots,a_{r+1}}} } > \tau,
\end{equation*}
which contradicts the inequality \eqref{1.2}. This completes the proof.
\end{proof}

\section{Generalizations of the $q$-analog of the LYM inequality}\label{sec3}
Let $\mathbb{F}_q$ be a finite field with $q$ elements, and let $L_q(n)=L(\mathbb{F}_q^n)$ denote the subspace lattice of $\mathbb{F}_q^n$. Within this lattice, it is known that any $k$-dimensional subspace has exactly $q^{k(n-k)}$ complements; see \cite[Lemma 8]{Beck and Zaslavsky2}.
Notably, the number of $k$-dimensional subspaces in $L_q(n)$ is captured by the $q$-analog of the binomial coefficients, which is called the $q$-Gaussian coefficients and defined as
\begin{equation*}
	\begin{bmatrix}n\\k\end{bmatrix}_{q}=\frac{[n]_{q}!}{[k]_q![n-k]_q!},\;\;\text{where}\;\;[n]_{q}!=(q^n-1)(q^{n-1}-1)\cdots(q-1).
\end{equation*}
Generalizing this notion, consider non-negative integers $a_1, a_2, \ldots , a_r$ such that $a_1 + \cdots + a_r = n$. The $q$-analog of the multinomial coefficients is
\begin{equation}\label{q-Gaussian}
	\genfrac{[}{]}{0pt}{}{n}{a_1,\ldots,a_r}_{q}=\frac{[n]_q!}{[a_1]_{q}!\cdots [a_r]_{q}!}.
\end{equation}
Recall that an $r$-decomposition of $\mathbb{F}_q^n$ is a tuple $(D_1,\ldots,D_r)$ of subspaces such that $D_1 \oplus \cdots \oplus D_r = \mathbb{F}_q^n$. 
Given a family $\mathcal{D}$ of $r$-decompositions of $\mathbb{F}_q^n$, we write 
\[
\mathcal{D}_k=\left\{D_{k}\mid (D_{1},\ldots,D_k, \ldots, D_{r})\in\mathcal{D}\right\}
\]
for each $k \in \{1, \ldots, r\}$, and
\[
	\mathcal{D}_{a_{1},\dots,a_{r}}= \{(D_1, \ldots, D_r)\in \mathcal{D}\mid \dim(D_i) = a_i\text{ for all } i=1,\ldots, r\}.
\]
It follows from \cite[Lemma 9]{Beck and Zaslavsky2} that when $\mathcal{D}$ is the set of all $r$-decompositions, we have
\begin{equation*}\label{11}
	|\mathcal{D}_{a_1, \ldots, a_r}| ={\textstyle \genfrac{[}{]}{0pt}{}{n}{a_1,\ldots,a_r}_{q}}\prod\limits_{i<j}q^{a_ia_j} .
\end{equation*}

\begin{Theorem}\label{Decq}
	Let  $t_1, \ldots, t_r$ be positive integers and $\sigma = \frac{t_1t_2\cdots t_r}{\mathrm{max}\{t_1, \ldots, t_r\}}$. Suppose $\mathcal{D}$ is a family of $r$-decompositions of $\mathbb{F}_q^n$ such that $\mathcal{D}_k$ is \emph{$t_k$-chain free} for each $k\in \{1, \ldots, r\}$. Then
	\begin{equation}\label{4.1}
		\sum_{a_1 + \cdots + a_r = n}\frac{|\mathcal{D}_{a_{1},\dots,a_{r}}|}{\genfrac{[}{]}{0pt}{}{n}{a_1,\ldots,a_r}_{q}\prod\limits_{i<j}q^{a_ia_j}}\le \sigma.
	\end{equation}
	Consequently,
	\begin{equation*}
		|\mathcal{D}|\le m_1+\cdots+m_\sigma,
	\end{equation*}
	where $m_1,\ldots,m_\sigma$ are the $\sigma$ largest $\genfrac{[}{]}{0pt}{}{n}{a_1,\ldots,a_r}_{q}\prod\limits_{i<j}q^{a_ia_j}$ for non-negative integers $a_1,\ldots,a_r$ with $a_1+\cdots + a_r=n$.
\end{Theorem}

\begin{proof}[Proof of theorem \ref{Decq}]
	Without loss of generality, assume $t_{r}=\mathrm{max}\{t_1, \ldots, t_r\}$.
	We proceed by induction on $r$. For $r = 2$, the inequality \eqref{4.1} reduces to
	\begin{equation}\label{d1}
		\begin{aligned}
			\sum_{a_1 +  a_2 = n}\frac{|\mathcal{D}_{a_{1},a_{2}}|}{\genfrac{[}{]}{0pt}{}{n}{a_1,a_2}_{q}q^{a_1a_2} }
			&=\sum_{(D_1,  D_2)\in\mathcal{D}}\frac{1}{\genfrac{[}{]}{0pt}{}{n}{\dim(D_1),  n-\dim(D_1)}_{q}q^{\dim(D_1)(n-\dim(D_1)}}\\
			&=\sum_{D\in \mathcal{D}_1}\frac{1}{\genfrac{[}{]}{0pt}{}{n}{\dim(D)}_{q}q^{\dim(D)(n-\dim(D))}}\sum\limits_{\substack{{ (D,  D_2)\in \mathcal{D}}}}1.
		\end{aligned}
	\end{equation}
	As any $k$-dimensional subspace in $L_q(n)$ has $q^{k(n-k)}$ complements, it follows that
	\[
	\sum\limits_{\substack{{ (D,  D_2)\in \mathcal{D}}}}1 \leq q^{\dim(D)(n-\dim(D))}.
	\]
	Therefore, by the inequality \eqref{a4} in the Rota-Harper theorem, the last expression in \eqref{d1} is at most $t_1$. Now suppose that $r> 2$. Note that for any $D \in \mathcal{D}_1$, by the induction hypothesis we have
	\begin{equation*}
		\begin{aligned}
				&\sum\limits_{\substack{{ (D, D_2, \ldots, D_r)\in \mathcal{D}} }}\frac{1}{\genfrac{[}{]}{0pt}{}{n-\dim(D)}{\dim(D_2), \ldots, \dim(D_r)}_{q}\prod\limits_{2\le i<j}q^{\dim(D_i)\dim(D_j)}}\\
			&=\sum\limits_{F \oplus D = \mathbb{F}_q^n}\sum\limits_{\substack{ (D, D_2, \ldots, D_r)\in \mathcal{D} \\ D_2 \oplus \cdots \oplus D_r = F}}\frac{1}{\genfrac{[}{]}{0pt}{}{\dim(F)}{\dim(D_2), \ldots, \dim(D_r)}_{q}\prod\limits_{2\le i<j}q^{\dim(D_i)\dim(D_j)}}\\
			&\leq q^{\dim(D)(n-\dim(D))} t_2\cdots t_{r-1}.
		\end{aligned}
	\end{equation*}
	Therefore,
		\begin{equation*}
		\begin{aligned}
			&\sum_{a_1 + \cdots + a_r = n}\frac{|\mathcal{D}_{a_{1},\dots,a_{r}}|}{\genfrac{[}{]}{0pt}{}{n}{a_1,\ldots,a_r}_{q} \prod\limits_{i<j}q^{a_ia_j}}\\
			&=\sum_{(D_1, \ldots, D_r)\in\mathcal{D}}\frac{1}{\genfrac{[}{]}{0pt}{}{n}{\dim(D_1), \ldots, \dim(D_r)}_{q}\prod\limits_{i<j}q^{\dim(D_i)\dim(D_j)}}\\
			&=\sum_{D\in \mathcal{D}_1}\frac{1}{\genfrac{[}{]}{0pt}{}{n}{\dim(D)}_{q}q^{\dim(D)(n-\dim(D))} }\sum\limits_{{ (D,D_2,  \ldots, D_r)\in \mathcal{D}}}\frac{1}{\genfrac{[}{]}{0pt}{}{n-\dim(D)}{\dim(D_2), \ldots, \dim(D_r)}_{q}\prod\limits_{2\le i<j}q^{\dim(D_i)\dim(D_j)}}\\
			&\leq t_2\cdots t_{r-1}\sum_{D\in \mathcal{D}_1}\frac{1}{\genfrac{[}{]}{0pt}{}{n}{\dim(D)}_{q}}.
		\end{aligned}
	\end{equation*}
	Again by \eqref{a4} in the Rota-Harper Theorem, the last expression is at most $\sigma = t_1\cdots t_{r-1}$, which completes the proof of the inequality \eqref{4.1}.
	
	Note that
	\[
	\mathcal{D} = \bigsqcup_{a_{1}+\cdots+a_{r}=n}\mathcal{D}_{a_{1},\dots,a_{r}}
	\]
	and $0\le|\mathcal{D}_{a_{1},\dots,a_{r}}|\le \genfrac{[}{]}{0pt}{}{n}{a_1,\ldots,a_r}_{q}\prod_{i<j}q^{a_ia_j}$ for each $\mathcal{D}_{a_{1},\dots,a_{r}}$. 
	When $\sigma>\binom{n+r-1}{r-1}$, we assume that $m_i = 0$ for all $i > \binom{n+r-1}{r-1}$ and then $|\mathcal{D}|\le m_1+\cdots+m_\sigma$ trivially holds. For the case $\sigma \leq \binom{n+r-1}{r-1}$, suppose for contradiction that $|\mathcal{D}| > m_1+\cdots+m_\sigma$.
	Then, by Lemma \ref{d} we have
	\begin{equation*}
		\sum_{a_{1}+\cdots+a_{r}=n} \frac{|\mathcal{D}_{a_{1},\dots,a_{r}}|}{\genfrac{[}{]}{0pt}{}{n}{a_1,\ldots,a_r}_{q}\prod\limits_{i<j}q^{a_ia_j}} > \sigma,
	\end{equation*}
	which contradicts the inequality \eqref{4.1}. This completes the proof.
\end{proof}

An \emph{$r$-multichain} in the subspaces lattice $L_q(n)$ is a tuple $(C_1, \ldots, C_r)$ of elements of $L_q(n)$ such that $C_1 \subseteq \cdots \subseteq C_r$.
Let $\mathcal{C}$ be a family of such $r$-multichains, and let $a_1, \ldots, a_r$ be non-negative integers with $a_1 + \cdots + a_r \leq n$. We define 
$$\mathcal{C}_k := \{C_k \mid (C_1, \ldots, C_k, \ldots, C_r) \in \mathcal{C}\}$$
for each $k \in \{1, \ldots, r\}$, and
\[
\mathcal{C}_{a_1,\dots,a_r} := \left\{ (C_1, \ldots, C_r) \in \mathcal{C} \mid \operatorname{rk}(C_i) = a_1 + \cdots + a_i \text{ for all } i = 1, \ldots, r \right\}.
\]
\begin{Theorem}\label{Chq}
	Let $t_1, \ldots, t_r$ be positive integers and $\tau= t_1t_2\cdots t_r$. Suppose $\mathcal{C}$ is a family of $r$-multichains in $L_q(n)$ such that $\mathcal{C}_k$ is $t_k$-chain free for each $k \in \{1, \ldots, r\}$. Then
	\begin{equation}\label{chq1}
		\sum_{a_1 + \cdots + a_{r+1} = n}\frac{|\mathcal{C}_{a_{1},\dots,a_{r}}|}{\genfrac{[}{]}{0pt}{}{n}{a_1,\ldots,a_{r+1}}_{q}}\le \tau.
	\end{equation}
	Consequently, 
	\[
	|\mathcal{C}| \leq m_1+\cdots+m_\tau,
	\]
	where $m_1,\ldots,m_\tau$ are the $\tau$ largest $\genfrac{[}{]}{0pt}{}{n}{a_1,\ldots,a_{r+1}}_{q}$ for non-negative integers $a_1,\ldots,a_{r+1}$ with $a_1+\cdots + a_{r+1}=n$.
\end{Theorem}

\begin{proof}[Proof of Theorem \ref{Chq}]
	We proceed by induction on $r$. For $r=1$, the inequality \eqref{chq1} reduces to the inequality \eqref{a4} stated in the Rota-Harper Theorem. Suppose that $r>1$ and that the inequality \eqref{chq1} holds for $r-1$. Then
	\begin{equation*}
		\begin{aligned}
			\sum_{a_1 + \cdots + a_{r+1} = n}\frac{|\mathcal{C}_{a_{1},\dots,a_{r}}|}{\genfrac{[}{]}{0pt}{}{n}{a_1,\ldots,a_{r+1}}_{q}}
			&=\sum_{(C_1, \ldots, C_r)\in\mathcal{C}}\frac{1}{\genfrac{[}{]}{0pt}{}{n}{\dim(C_1),\dim(C_2)-\dim(C_1) ,\ldots,n-\dim(C_r)}_{q}}\\
			&=\sum_{C\in \mathcal{C}_1}\frac{1}{\genfrac{[}{]}{0pt}{}{n}{\dim(C)}_{q}}\sum\limits_{\substack{{ (C,C_2, \ldots, C_r)\in \mathcal{C}} }}\frac{1}{\genfrac{[}{]}{0pt}{}{n-\dim(C)}{\dim(C_2)-\dim(C_1), \ldots,n-\dim(C_r)}_{q}}\\
			&\leq  \sum_{C\in \mathcal{C}_1}\frac{1}{\genfrac{[}{]}{0pt}{}{n}{\dim(C)}_{q}} t_2\cdots t_{r}.\\
		\end{aligned}
	\end{equation*}
	Again, by \eqref{a4}, the last expression is at most $\tau = t_1\cdots t_{r}$.
	
	Note that
	\[
	\mathcal{C} = \bigsqcup_{a_{1}+\cdots+a_{r+1}=n}\mathcal{C}_{a_{1},\dots,a_{r}}
	\]
	and $0\le|\mathcal{C}_{a_{1},\dots,a_{r}}|\le \genfrac{[}{]}{0pt}{}{n}{a_1,\ldots,a_{r+1}}_{q}$ for each $\mathcal{C}_{a_{1},\dots,a_{r}}$. When $\tau >\binom{n+r}{r}$, we assume that $m_i = 0$ for all $i > \binom{n+r}{r}$ and then $|\mathcal{C}|\le m_1+\cdots+m_\tau$ trivially holds. For the case $\tau \leq \binom{n+r}{r}$, suppose for contradiction that $|\mathcal{C}| > m_1+\cdots+m_\tau$.
	Then, by Lemma \ref{d} we have
	\begin{equation*}
		\sum_{a_{1}+\cdots+a_{r+1}=n} \frac{|\mathcal{C}_{a_{1},\dots,a_{r}}|}{\genfrac{[}{]}{0pt}{}{n}{a_1,\ldots,a_{r+1}}_{q}} > \tau,
	\end{equation*}
	which contradicts the inequality \eqref{chq1}. This completes the proof.
\end{proof}

\section{Generalizations of the continuous analog of the LYM inequality}\label{sec4}

Let $\mathrm{Mod}(n)$ denote the subspace lattice of $\mathbb{R}^n$. For each $k \in \{0,\ldots, n\}$, the Grassmannian $\mathrm{Gr}(n, k)$ is the set of all $k$-dimensional subspaces in $\mathrm{Mod}(n)$, which is known as the Grassmann manifold. A \emph{flag} in $\mathrm{Mod}(n)$ is a tuple $F = (F_0, \ldots, F_n)$ of elements of $\mathrm{Mod}(n)$, such that $F_k \in \mathrm{Gr}(n,k)$ for $k \in \{0,\ldots, n\}$ and $
F_0 \subseteq F_1 \subseteq \cdots \subseteq F_n.$ Denote by $\mathrm{Flag}(n)$ the set of all flags in $\mathrm{Mod}(n)$, which is also called the flag manifold. 

Let $O(n)$ denote the $n$-dimensional orthogonal group, which naturally acts on $\mathrm{Gr}(n,k)$ and $\mathrm{Flag}(n)$ via matrix multiplication on $\mathbb{R}^n$. Both $\mathrm{Gr}(n,k)(k \geq 1)$ and $\mathrm{Flag}(n)$ carry a unique $O(n)$-invariant measure, known as the Haar measures \cite{Haar, Haar2} and denoted by $\nu_k^n$ and $\phi_n$, respectively.
To make this concrete, let $\omega_{n}$ denote the volume of the unit ball in $\mathbb{R}^n$, and let $S^{n-1}$ be the unit sphere with its standard $O(n)$-invariant measure denoted by $\rho_{n-1}$. The measure $\nu_1^n$ on $\mathrm{Gr}(n,1)$ is defined by
\[
\nu_1^n(A) = \frac{\rho_{n-1}(\bigcup_{x \in A}x \cap S^{n-1})}{\omega_{n-1}},
\]
for any subset $A \subseteq \mathrm{Gr}(n,1)$.
By orthogonal duality, $\nu_1^n$ induces the measure $\nu_{n-1}^n$ on $\mathrm{Gr}(n, n-1)$.
Inductively, this duality extends to a measure $\phi_n$ on $\mathrm{Flag}(n)$: for any simple function $f(F_0, F_1, \ldots, F_n)$ on $\mathrm{Flag}(n)$,
\[
\int f d\phi_n = \int\int f(F_0, F_1, \ldots, F_n)d\phi_{n-1}(F_0, \ldots, F_{n-1})d\nu_{n-1}^n(F_{n-1}).
\]
Define $[n]_{\mathbb{R}} := \dfrac{n\omega_{n}}{2\omega_{n-1}}$.
The total measure of $\mathrm{Flag}(n)$ turns out to be
\[
\phi_n\big(\mathrm{Flag}(n)\big) = [n]_{\mathbb{R}}! := [n]_{\mathbb{R}}[n-1]_{\mathbb{R}}\cdots[1]_{\mathbb{R}} =  \dfrac{n! \omega_n}{2^n}.
\]
Using $\phi_n$, we induce the measure $\nu_{k}^{n}$ on $\mathrm{Gr}(n,k)$ by, for any subset $A \subseteq \mathrm{Gr}(n, k)$,
\[
\nu_{k}^{n}(A) = \frac{1}{[k]_{\mathbb{R}}![n-k]_{\mathbb{R}}!}\phi_n\big(\mathrm{Flag}(A)\big),
\]
where $\mathrm{Flag}(A) \subseteq \mathrm{Flag}(n)$ denotes the subset consisting of all flags $(F_0, \ldots, F_n)$ such that $F_k \in A$.
In particular, we have
\begin{equation}\label{con_nk}
	\nu_{k}^{n}\big(\mathrm{Gr}(n, k)\big) = \dfrac{[n]_{\mathbb{R}}!}{[k]_{\mathbb{R}}![n-k]_{\mathbb{R}}!} = \binom{n}{k} \frac{\omega_{n}}{\omega_{k}\omega_{n-k}},
\end{equation}
which is also denoted by $\genfrac{[}{]}{0pt}{}{n}{k}_{\mathbb{R}}$.

We now turn to the continuous analog of the LYM inequality. An \emph{$r$-decomposition} of $\mathbb{R}^{n}$ is a tuple $(D_1, \ldots, D_r)$ of nonzero subspaces in $\mathrm{Mod}(n)$ satisfying $\bigoplus_{i=1}^r D_i = \mathbb{R}^n$ and $D_i \perp D_j$ for all $i \neq j$. We explicitly exclude the zero subspace from this definition, since $\mathrm{Gr}(n,0)$ is not compatible with the recursive definition of the Haar measure on $\mathrm{Gr}(n,k)$ for $k \geq 1$.
Denote by $\mathcal{D}(\mathbb{R}^n,r)$ the set of all $r$-decompositions of $\mathbb{R}^{n}$. Given positive integers $a_1, \ldots, a_r$ such that $a_1 + \cdots + a_r =n$, let $\mathcal{D}(\mathbb{R}^n, r)_{a_1,a_2,\ldots,a_r}$ be the set of all $r$-decompositions $(D_1, \ldots, D_r) \in \mathcal{D}(\mathbb{R}^n,r)$ such that $\dim(D_i)= a_i$ for $i= 1, \ldots, r$. Evidently we have the finite disjoint union
\[
\mathcal{D}(\mathbb{R}^n,r)=\bigsqcup_{a_{1}+\cdots+a_{r}=n}{\mathcal{D}(\mathbb{R}^n, r)_{a_1,a_2,\ldots,a_r}}.
\]
Each $\mathcal{D}(\mathbb{R}^n, r)_{a_1,a_2,\ldots,a_r}$ is also a compact smooth manifold, which carries the unique Haar measure and is denoted by $\nu_{a_1,a_2,\ldots,a_r}^{n}$. To construct this measure explicitly, we call a flag $(F_0, \ldots, F_n)$ \emph{compatible} with $(D_1, \ldots, D_r) \in\mathcal{D}(\mathbb{R}^n, r)_{a_1,a_2,\ldots,a_r}$ if $F_{a_1+\cdots+a_i} = D_1 \oplus \cdots \oplus D_i$ for all $i \in \{1, \ldots, r\}$.
Then the measure $\nu_{a_1,a_2,\ldots,a_r}^{n}$ is induced from $\phi_n$ on $\mathrm{Flag}(n)$ by, for any $\mathcal{D} \subseteq \mathcal{D}(\mathbb{R}^n, r)_{a_1,a_2,\ldots,a_r}$, 
\begin{equation}\label{val}
	\nu_{a_1,a_2,\ldots,a_r}^{n}(\mathcal{D})=\frac{1}{[a_1]_{\mathbb{R}}![a_2]_{\mathbb{R}}!\cdots[a_r]_{\mathbb{R}}!}\phi_n\big(\mathrm{Flag}(\mathcal{D})\big),
\end{equation}
where $\mathrm{Flag}(\mathcal{D})$ is the set of all flags compatible with some $r$-decomposition of $\mathcal{D}$.
In particular, the total measure is
\begin{equation}\label{evi}
	\nu_{a_1,a_2,\ldots,a_r}^{n}\big(\mathcal{D}(\mathbb{R}^n,r)_{a_1,\ldots,a_r}\big) = \frac{[n]_{\mathbb{R}}!}{[a_1]_{\mathbb{R}}!\cdots[a_r]_{\mathbb{R}}!}.
\end{equation}
These values are called \emph{multiflag coefficients} and denoted by  $\genfrac{[}{]}{0pt}{}{n}{a_1,\ldots,a_r}_{\mathbb{R}}$. Furthermore, the Haar measure $\nu_{n;r}$ on $\mathcal{D}(\mathbb{R}^n,r)$ is induced from $\nu_{a_1,a_2,\ldots,a_r}^{n}$ by aggregating the measures of its disjoint components. That is, for any subset $\mathcal{D} \subseteq \mathcal{D}(\mathbb{R}^n,r)$,
\[
\nu_{n;r}(\mathcal{D})=\sum_{a_{1}+\cdots+a_{r}=n}\nu_{a_1,\ldots,a_r}^n(\mathcal{D}_{a_1,\ldots,a_r}),
\]
where $\mathcal{D}_{a_1,\ldots,a_r}=\mathcal{D} \cap \mathcal{D}(\mathbb{R}^n,r)_{a_1,\ldots,a_r}$. For more details, we refer the reader to \cite{Klain and Rota}.

\begin{Theorem}\label{DecR}
	Let  $t_1, \ldots, t_r$ be positive integers and $\sigma = \frac{t_1t_2\cdots t_r}{\mathrm{max}  \{ t_{1} ,t_2,\ldots,t_r \}}$. Suppose $\mathcal{D}$ is a family of $r$-decompositions of $\mathbb{R}^n$ such that $\mathcal{D}_k$ is $t_k$-chain free for each $k \in \{1, \ldots, r\}$. Then
	\begin{equation}\label{2.1}
		\sum_{a_{1}+\cdots+a_{r}=n}\dfrac{\nu_{a_1,\ldots,a_r}^n(\mathcal{D}_{a_1,\ldots,a_r})}{\genfrac{[}{]}{0pt}{}{n}{a_1,\ldots,a_r}_{\mathbb{R}}}\le \sigma.
	\end{equation}
	Consequently, 
	\[
	\nu_{n;r}(\mathcal{D}) \leq m_1 + \cdots + m_\sigma,
	\]
	where $m_1, \ldots, m_\sigma$ are the $\sigma$ largest multiflag coefficients $\genfrac{[}{]}{0pt}{}{n}{a_1,\ldots,a_r}_{\mathbb{R}}$ for positive integers $a_1, \ldots, a_r$ with $a_1+\cdots + a_r=n$.
\end{Theorem}

\begin{proof}
	It is known that the flag manifold $\mathrm{Flag}(n)$ is a bundle over $\mathrm{Gr}(n,k)$ with projection map $\pi_k$,
	where $\pi_k(F_0,F_1,\dots,F_n)=F_k$ for all $(F_0,F_1,\dots,F_n)\in\mathrm{Flag}(n)$. See \cite[pp.~120--122]{Warner} and \cite[pp.~30--32]{Norman}. The fibre of the bundle is $\mathrm{Flag}(k)\times\mathrm{Flag}(n-k)$.
	For each $D\in\mathrm{Gr}(n,k)$, there exists an open neighborhood $U\subseteq\mathrm{Gr}(n,k)$ and a local trivialization homeomorphism $ \phi: \pi_{k}^{-1}(U) \longrightarrow U \times \mathrm{Flag}(k)\times\mathrm{Flag}(n-k) $ such that
	$ \pi_k = \text{proj}_U \circ \phi $, where $ \text{proj}_U $ denotes the projection onto the first coordinate.
	The measure on  $\pi_{k}^{-1}(U)$, which is homeomorphic to $ U \times \mathrm{Flag}(k)\times\mathrm{Flag}(n-k) $, is induced by the measure $\nu_{k}^{n}\times\phi_k\times\phi_{n-k}$ on $\mathrm{Gr}(n,k) \times \mathrm{Flag}(k)\times\mathrm{Flag}(n-k)$.
	
	Given positive integers $a_1,\cdots,a_r$ satisfying $a_1+\cdots+a_r=n$, recall that
	\[
	(\mathcal{D}_{a_{1},\ldots,a_{r}})_1 = \{ D \in \mathcal{D}_1 \mid \mathrm{dim} (X)=a_1 \}.
	\]
	Taking $D \in (\mathcal{D}_{a_{1},\ldots,a_{r}})_1$, we define
	\[
	\mathcal{D}_{D, a_2, \ldots, a_r} = \{(D_2, \ldots, D_r) \mid (D, D_2, \ldots, D_r) \in \mathcal{D}_{a_{1},\ldots,a_{r}}\}.
	\]
	Noting that $\big(\mathcal{D}(\mathbb{R}^n, r)\big)_{D, a_2, \ldots, a_r}$ is homeomorphic to $\big(\mathcal{D}(\mathbb{R}^{n-a_1}, r-1)\big)_{a_2, \ldots, a_r}$, we regard $\mathcal{D}_{D, a_2, \ldots,a_r}$ as a subset of $\big(\mathcal{D}(\mathbb{R}^{n-a_1}, r-1)\big)_{a_2, \ldots, a_r}$.
   	Therefore,
	\begin{equation*}
		\begin{aligned}
			&\phi_n\big(\mathrm{Flag}(\mathcal{D}_{a_1,\ldots,a_r})\big)\\
			&=\int_{\mathrm{Gr}(n,a_1)}\chi _{(\mathcal{D}_{a_{1},\ldots,a_{r}})_1}(D)\Big(\int_{\mathrm{Flag}(a_1)} 1 d\phi_{a_1}\Big)\phi_{n-a_1}\big(\mathrm{Flag}(\mathcal{D}_{D, a_2, \ldots, a_r})\big) d\nu_{a_1}^{n}\\[3pt]
			&=[a_1]!\int_{\mathrm{Gr}(n,a_1)}\chi _{(\mathcal{D}_{a_{1},\ldots,a_{r}})_1}(D)\phi_{n-a_1}\big(\mathrm{Flag}(\mathcal{D}_{D, a_2, \ldots, a_r})\big)d\nu_{a_1}^{n},
		\end{aligned}
	\end{equation*}
	where $\chi_A$ denotes the indicator function of the set $A$.
	It follows from \eqref{val} that
	\begin{equation*}
		\nu_{a_1,\ldots,a_r}^n(\mathcal{D}_{a_1,\ldots,a_r})=\int_{\mathrm{Gr}(n,a_1)}\chi_{(\mathcal{D}_{a_{1},\ldots,a_{r}})_1}(D)\nu_{a_2,\ldots,a_r}^{n-a_1}\big(\mathcal{D}_{D, a_2, \ldots, a_r}\big)d\nu_{a_1}^{n}.
	\end{equation*}
	Without loss of generality, assume that $t_{r}=\mathrm{max} \left \{ t_{1} ,t_2,\ldots,t_r \right \}$.
	For $r = 2$, the inequality \eqref{2.1} reduces to
	\begin{equation}\label{d2}
		\begin{aligned}
			\sum_{a_1 +  a_2 = n}\dfrac{\nu_{a_1,a_2}^n(\mathcal{D}_{a_1,a_2})}{\genfrac{[}{]}{0pt}{}{n}{a_1,a_2}_{\mathbb{R}}}
			&=\sum_{a_1=1}^{n-1}\dfrac{\int_{\mathrm{Gr}(n,a_1)}\chi_{(\mathcal{D}_{a_{1},n-a_{1}})_1}(D)\nu_{n-a_1}^{n-a_1}(\mathcal{D}_{D,n-a_1})d\nu_{a_1}^{n}}{\genfrac{[}{]}{0pt}{}{n}{a_{1}}_{\mathbb{R}}}.
		\end{aligned}
	\end{equation}
	Since $\nu_{k}^{k}(\mathbb{R}^k)=1$ for any $k \geq 1$. It follows that for each $D \in (\mathcal{D}_{a_{1},n-a_{1}})_1$, we have $\nu_{n-a_1}^{n-a_1}(\mathcal{D}_{D,n-a_1})=1$.
	Hence \eqref{d2} equals
	\[
	\sum_{a_1=1}^{n-1}\dfrac{\nu_{a_1}^{n}\big((\mathcal{D}_{a_1, n-a_1})_1\big)}{\genfrac{[}{]}{0pt}{}{n}{a_{1}}_{\mathbb{R}}},
	\]
	which is at most $t_1$ by \eqref{a6} in the Klain–Rota Theorem.
	
	Now suppose that $r>2$. In this case,
	\begin{equation*}
		\begin{aligned}
			&\sum_{a_{1}+\cdots+a_{r}=n}\dfrac{\nu_{a_1,\ldots,a_r}(\mathcal{D}_{a_1,\ldots,a_r})}{\genfrac{[}{]}{0pt}{}{n}{a_1,\ldots,a_r}_{\mathbb{R}}}\\
			&=\sum_{a_1=1}^{n-r+1}\sum_{a_2+\cdots+a_r=n-a_1}\dfrac{\int_{\mathrm{Gr}(n,a_1)}\chi_{(\mathcal{D}_{a_{1},\ldots,a_{r}})_1}(D)\nu_{a_2,\ldots,a_r}^{n-a_1}(\mathcal{D}_{D, a_2, \ldots, a_r})d\nu_{a_1}^{n}}{\genfrac{[}{]}{0pt}{}{n}{a_{1}}_{\mathbb{R}}\genfrac{[}{]}{0pt}{}{n-a_1}{a_2,\ldots,a_r}_{\mathbb{R}}}.
		\end{aligned}
	\end{equation*}
	It follows from the induction hypothesis that for every $D \in (\mathcal{D}_{a_1,\ldots,a_r})_1$,
	\[
	\sum_{a_{2}+\cdots+a_{r}=n-a_{1}}\dfrac{\nu_{a_2,\ldots,a_r}^{n-a_1}(\mathcal{D}_{D, a_2, \ldots, a_r})}{\genfrac{[}{]}{0pt}{}{n-a_1}{a_2,\ldots,a_r}_{\mathbb{R}}}\le t_2\cdots t_{r-1}.
	\]
	Thus, we have
	\[
	\sum_{a_{1}+\cdots+a_{r}=n}\dfrac{\nu_{a_1,\ldots,a_r}(\mathcal{D}_{a_1,\ldots,a_r})}{\genfrac{[}{]}{0pt}{}{n}{a_1,\ldots,a_r}_{\mathbb{R}}}
	\le\sum_{a_1=1}^{n-r+1}\dfrac{\nu_{a_1}^{n}\big((\mathcal{D}_{a_1, \ldots, a_r})_1\big)}{\genfrac{[}{]}{0pt}{}{n}{a_{1}}_{\mathbb{R}}}t_2\cdots t_{r-1} \leq t_1\cdots t_{r-1} = \sigma.
	\]
	
	Note that
	\[
	\mathcal{D}(\mathbb{R}^n,r)=\bigsqcup_{a_{1}+\cdots+a_{r}=n}\mathcal{D}(\mathbb{R}^n, r)_{a_1,a_2,\ldots,a_r}
	\]
	and $0\le\nu_{a_1,\ldots,a_r}^n(\mathcal{D}_{a_1,\ldots,a_r})\le \genfrac{[}{]}{0pt}{}{n}{a_1,\ldots,a_r}_{\mathbb{R}}$ for each $\mathcal{D}_{a_{1},\dots,a_{r}}$ by \eqref{evi}. When $\sigma > \binom{n-1}{r-1}$, we assume that $m_i = 0$ for all $i > \binom{n-1}{r-1}$ and then $\nu_{n;r}(\mathcal{D}) \le m_1+\cdots+m_\sigma$ trivially holds. For the case $\sigma \leq \binom{n-1}{r-1}$, suppose for contradiction that $\nu_{n;r}(\mathcal{D}) > m_1 + \cdots + m_\sigma$.
	Then, by Lemma \ref{d},
	\begin{equation*}
		\sum_{a_{1}+\cdots+a_{r}=n}\dfrac{\nu_{a_1,\ldots,a_r}^n(\mathcal{D}_{a_1,\ldots,a_r})}{\genfrac{[}{]}{0pt}{}{n}{a_1,\ldots,a_r}_{\mathbb{R}}} > \sigma,
	\end{equation*}
	which contradicts the inequality \eqref{2.1}. This completes the proof.
	\end{proof}
	
	An $r$-chain of $\mathrm{Mod}(n)$ is a tuple $(C_1, \ldots, C_r)$ of nonzero subspaces in $\mathbb{R}^n$ such that $C_1 \subsetneq \cdots \subsetneq C_r$. We denote by $\mathcal{C}(\mathbb{R} ^{n},r)$ the set of all $r$-chains of $\mathrm{Mod}(n)$. Given positive integers $a_1, \ldots, a_r$ with $a_1 + \cdots + a_r < n$, denote by $\mathcal{C}(\mathbb{R} ^{n},r)_{a_1,\ldots,a_r}$ the subset consisting of all $r$-chains $(C_1, \ldots, C_r)$ such that $\dim(C_i)=a_1+a_2+\cdots+a_i$ for all $i= 1, \ldots, r$. $\mathcal{C}(\mathbb{R} ^{n},r)_{a_1,\ldots,a_r}$ is also a compact smooth manifold equipped with a unique $O(n)$-invariant measure $\mu_{a_1,\ldots,a_r}^n$. More precisely, we call a flag $(F_0, \ldots, F_n)$ of $\mathrm{Mod}(n)$ \emph{compatible} with $(C_1,\ldots,C_r) \in \mathcal{C}(\mathbb{R} ^{n},r)_{a_1,\ldots,a_r}$ if $F_{a_1+\cdots+a_i} = C_i$ for all $i \in \{1, \ldots, r\}$. Then the measure $\mu_{a_1,\ldots,a_r}^n$ is induced from the measure $\phi_n$ on $\mathrm{Flag}(n)$ by
	\begin{equation}\label{chain}
	\mu_{a_1,a_2,\ldots,a_r}^{n}(\mathcal{C})=\frac{1}{[a_1]_{\mathbb{R}}!\cdots[a_r]_{\mathbb{R}}![a_{r+1}]_{\mathbb{R}}!}\phi_n\big(\mathrm{Flag}(\mathcal{C})\big)
\end{equation}
	for any $\mathcal{C} \subseteq \mathcal{C}(\mathbb{R} ^{n},r)_{a_1,\ldots,a_r}$, where $a_{r+1} = n-\sum_{i=1}^{r}a_i$, and $\mathrm{Flag}(\mathcal{C})$ represents the subset of $\mathrm{Flag}(n)$ consisting of all flags compatible with some $r$-chain of $\mathcal{C}$.
	In particular, The total measure of $\mathcal{C}(\mathbb{R} ^{n},r)_{a_1,\ldots,a_r}$ is given by
	\[
	\mu_{a_1,a_2,\ldots,a_r}^{n}\big(\mathcal{C}(\mathbb{R} ^{n},r)_{a_1,\ldots,a_r}\big) = \genfrac{[}{]}{0pt}{}{n}{a_1,\ldots,a_{r}, a_{r+1}}_{\mathbb{R}}.
	\]
	 For any $\mathcal{C} \subseteq \mathcal{C}(\mathbb{R}^n, r)$, let $\mathcal{C}_{a_1,\ldots,a_r} = \mathcal{C} \cap \mathcal{C}(\mathbb{R} ^{n},r)_{a_1,\ldots,a_r}$ 
	 and $$\mathcal{C}_k = \{C_k \mid (C_1,\ldots,C_k,\ldots,C_r) \in \mathcal{C}\}.$$ 
	 These measures $\mu_{a_1,a_2,\ldots,a_r}^{n}$ induce the measure $\mu_{n;r}$ on $\mathcal{C}(\mathbb{R}^n,r)$ by
	\[
	\mu_{n;r}(\mathcal{C})=\sum_{a_{1}+\cdots+a_{r+1}=n}\mu_{a_1,\ldots,a_r}^n(\mathcal{C}_{a_1,\ldots,a_r})
	\]
	for any subset $\mathcal{C} \subseteq \mathcal{C}(\mathbb{R}^n,r)$.

\begin{Theorem}\label{ChR}
	Let  $t_1, \ldots, t_r$ be positive integers and $\tau = t_1t_2 \cdots t_r$. Suppose $\mathcal{C} $ is a family of $r$-chains in $\mathrm{Mod}(n)$ such that $\mathcal{C}_k$ is $t_k$-chain free for each $k \in \{1, \ldots, r\}$. Then
	\begin{equation}\label{2.2}
		\sum_{a_{1}+\cdots+a_{r+1}=n}\dfrac{\mu_{a_1,\ldots,a_r}^n(\mathcal{C}_{a_1,\ldots,a_r})}{\genfrac{[}{]}{0pt}{}{n}{a_1,\ldots,a_{r+1}}_{\mathbb{R}}}\le \tau.
	\end{equation}
	Consequently, 
	\[
	\mu_{n;r}(\mathcal{C})\leq m_1 +\cdots + m_\tau,
	\]
	where $m_1, \ldots, m_\tau$ are the $\tau$ largest multiflag coefficents $\genfrac{[}{]}{0pt}{}{n}{a_1,\ldots,a_{r+1}}$ for positive integers $a_1, \ldots, a_{r+1}$ with $a_1+\cdots + a_{r+1}=n$.
\end{Theorem}

\begin{proof}
	Given positive integers $a_1,\cdots,a_r$ satisfying $a_1+\cdots+a_r < n$, recall that
	\[
	(\mathcal{C}_{a_{1},\ldots,a_{r}})_1 = \left\{ C \in \mathcal{C}_1 \mid \mathrm{dim} (C)=a_1 \right\}.
	\]
	For any $C \in (\mathcal{C}_{a_{1},\ldots,a_{r}})_1$, we define
	\[
	\mathcal{C}_{C, a_2, \ldots, a_r} = \{(C_2, \ldots, C_r) \mid (C, C_2, \ldots, C_r) \in \mathcal{C}_{a_{1},\ldots,a_{r}}\}.
	\]
	Then $\mathcal{C}(\mathbb{R}^n, r)_{C, a_2, \ldots, a_r}$ is homeomorphic to $\mathcal{C}(\mathbb{R}^{n-a_1}, r-1)_{a_2, \ldots, a_r}$. Therefore, by regarding $\mathcal{C}_{C, a_2, \ldots,a_r}$ as a subset of $\mathcal{C}(\mathbb{R}^{n-a_1}, r-1)_{a_2, \ldots, a_r}$, we deduce from \eqref{chain} that
	\begin{equation*}
		\mu_{a_1,\ldots,a_r}^n(\mathcal{C}_{a_1,\ldots,a_r})=\int_{\mathrm{Gr}(n,a_1)}\chi_{(\mathcal{C}_{a_{1},\ldots,a_{r}})_1}(C)\mu_{a_2,\ldots,a_r}^{n-a_1}(\mathcal{C}_{C, a_2, \ldots, a_r})d\nu_{a_1}^{n}.
	\end{equation*}
	Similar to the proof of Theorem \ref{DecR}.
	We proceed by induction on $r$. For $r=1$, the inequality \eqref{2.2} reduces to the inequality \eqref{a6} in the Klain-Rota Theorem. Suppose that $r>1$ and that the inequality \eqref{2.2} holds for $r-1$. Then,
	\begin{equation*}
		\begin{aligned}
			&\sum_{a_{1}+\cdots+a_{r+1}=n}\dfrac{\mu_{a_1,\ldots,a_r}(\mathcal{C}_{a_1,\ldots,a_r})}{\genfrac{[}{]}{0pt}{}{n}{a_1,\ldots,a_{r+1}}_{\mathbb{R}}}\\
			&=\sum_{a_1=1}^{n-r}\sum_{a_2+\cdots+a_{r+1}=n-a_1}
			\dfrac{\int_{\mathrm{Gr}(n,a_1)}\chi_{(\mathcal{C}_{a_{1},\ldots,a_{r}})_1}(C)\mu_{a_2,\ldots,a_r}^{n-a_1}(\mathcal{C}_{C, a_2, \ldots, a_r})d\nu_{a_1}^{n}}{\genfrac{[}{]}{0pt}{}{n}{a_{1}}_{\mathbb{R}}\genfrac{[}{]}{0pt}{}{n-a_1}{a_2,\ldots,a_{r+1}}_{\mathbb{R}}}.
		\end{aligned}
	\end{equation*}
	It follows from the induction hypothesis that
	\[
	\sum_{a_{2}+\cdots+a_{r+1}=n-a_{1}}\dfrac{\mu_{a_2,\ldots,a_r}^{n-a_1}(\mathcal{C}_{C, a_2, \ldots, a_r})}{\genfrac{[}{]}{0pt}{}{n-a_1}{a_2,\ldots,a_{r+1}}_{\mathbb{R}}}\le t_2\cdots t_{r}
	\]
	for any $C \in (\mathcal{C}_{a_{1},\ldots,a_{r}})_1$.
	Thus we have
	\[
	\sum_{a_{1}+\cdots+a_{r+1}=n}\dfrac{\mu_{a_1,\ldots,a_r}(\mathcal{C}_{a_1,\ldots,a_r})}{\genfrac{[}{]}{0pt}{}{n}{a_1,\ldots,a_{r+1}}_{\mathbb{R}}}
	\le\sum_{a_1=1}^{n-r}\dfrac{\nu_{a_1}^{n}\big((\mathcal{C}_{a_1, \ldots, a_r})_1\big)}{\genfrac{[}{]}{0pt}{}{n}{a_{1}}_{\mathbb{R}}}t_2\cdots t_{r} \leq t_1\cdots t_{r} = \tau.
	\]
	
	Note that
	\[
	\mathcal{C}(\mathbb{R}^n,r)=\bigsqcup_{a_{1}+\cdots+a_{r+1}=n}{\mathcal{C}(\mathbb{R}^n,r)_{a_1,\ldots,a_r}}
	\]
	and $0\le\mu_{a_1,\ldots,a_r}^n(\mathcal{C}_{a_1,\ldots,a_r})\le \genfrac{[}{]}{0pt}{}{n}{a_1,\ldots,a_{r+1}}_{\mathbb{R}}$ for each $\mathcal{C}_{a_{1},\dots,a_{r}}$. The inequality $\mu_{n;r}(\mathcal{C}) \leq m_1+\cdots+m_\tau$ clearly holds when $\tau > \binom{n-1}{r}$, assuming that $m_i = 0$ for all $i > \binom{n-1}{r}$. For the case $\tau \leq \binom{n-1}{r}$, suppose for contradiction that $\mu_{n;r}(\mathcal{C}) > m_1 + \cdots + m_\tau$.
	Thus, by Lemma \ref{d},
	\begin{equation*}
		\sum_{a_{1}+\cdots+a_{r+1}=n}\dfrac{\mu_{a_1,\ldots,a_r}^n(\mathcal{C}_{a_1,\ldots,a_r})}{\genfrac{[}{]}{0pt}{}{n}{a_1,\ldots,a_{r+1}}_{\mathbb{R}}} >\tau,
	\end{equation*}
	which contradicts the inequality \eqref{2.2}. This completes the proof.
\end{proof}

\section{Generalizations of the arithmetic analog of the LYM inequality}\label{sec-5}
Let $n$ be a positive integer. By the fundamental theorem of arithmetic, there exists a unique expression $n = p_1^{e_1} \cdots p_m^{e_m}$, where $p_1, \ldots, p_m$ are distinct primes and $e_1, \ldots, e_m$ are positive integers. Let $\mathrm{Div}(n)$ denote the set of all positive divisors of $n$, partially ordered by divisibility. Then $\mathrm{Div}(n)$ forms a lattice. For any $x = p_1^{\ell_1} \cdots p_m^{\ell_m} \in \mathrm{Div}(n)$, the rank of \( x \) is given by $\rk(x) = \sum_{i=1}^{m} \ell_i$.
Building on the inequality \eqref{a8} in Anderson's theorem, the following lemma serves as a crucial step in proving our main result.

\begin{Lemma}\label{tch}
	For any $t$-chain free family $\A$ of elements of $\mathrm{Div}(n)$, we have
	\begin{equation*}
		\sum_{i=0}^{\rk(n)} \frac{|\mathcal{A}_i|}{W_i(n)} \le t,
	\end{equation*}
	where $\mathcal{A}_i$ denotes the elements of $\mathcal{A}$ of rank $i$.
\end{Lemma}
\begin{proof}
	 According to Mirsky's Theorem \cite[Theorem 2]{Mirsky}, which states that any $t$-chain free family of a poset can be expressed as the disjoint union of $t$ antichains, we write $\mathcal{A} = \mathcal{A}^{(1)} \cup \cdots \cup \mathcal{A}^{(t)}$ such that each $\mathcal{A}^{(i)}$ is an antichain. Therefore, by \eqref{a8} in Anderson's Theorem,
	\begin{equation*}
			\sum_{i=0}^{\rk(n)} \frac{|\mathcal{A}_i|}{W_i(n)} = \sum_{x\in \mathcal{A}}\frac{1}{W_{\rk(x)}(n)}
			=\sum_{i=1}^{t}\sum_{x\in \mathcal{A}^{(i)}}\frac{1}{W_{\rk(x)}(n)}
			\le \sum_{i=1}^{t}1
			=t.
	\end{equation*}
\end{proof}

Recall that an \emph{$r$-decomposition} of $n$ is an $r$-tuple $\left(x_1, \ldots, x_r\right)$ of $\mathrm{Div}(n)$ such that $x_1\cdots x_r=n$. Denote by $\mathcal{D}(n,r)$ the set of all $r$-decompositions of $n$.
Let $\mathcal{D} \subseteq \mathcal{D}(n,r)$ be a subset, and $a_1, \ldots, a_r$ be non-negative integers such that $a_1 + \cdots + a_r = \rk(n)$. We write
\[
\mathcal{D}_k:=\left\{x_k\mid \left(x_1, \ldots, x_r\right)\in\mathcal{D}\right\}
\]
for each $k \in \{1, \ldots, r\}$, and 
\[
\mathcal{D}_{a_1, \ldots, a_r}=\left\{\left(x_1, \ldots, x_r\right)\in\mathcal{D}\mid \rk(x_i)=a_i,\;i=1, \ldots, r\right\}.
\]
Furthermore, we denote by $N_{a_1,\ldots,a_r}(n) = |\mathcal{D}(n, r)_{a_1, \ldots, a_r}|$. 

\begin{Theorem}\label{AP}
	Let $t_1, \ldots, t_r$ be positive integers and $\sigma = \frac{t_1t_2\cdots t_r}{\mathrm{max}\{t_1, \ldots, t_r\}}$. Suppose $\mathcal{D}$ is a family of $r$-decompositions of $n$ such that $\mathcal{D}_k$ is $t_k$-chain free for each $k \in \{1, \ldots, r\}$. Then
	\begin{equation}\label{5.1}
		\sum_{a_{1}+\cdots+a_{r}=\rk(n)}^{} \frac{|\mathcal{D}_{a_{1},\dots,a_{r}}|}{N_{a_1, \ldots,a_r}(n)} \leq \sigma.
	\end{equation}
	Consequently,
	\[
	|\mathcal{D}| \leq m_1 + \cdots + m_\sigma,
	\]
	where $m_1,\ldots,m_\sigma$ are the $\sigma$ largest $N_{a_1,\ldots,a_r}(n)$ for non-negative integers $a_1,\ldots,a_r$ with $a_1+\cdots + a_r=\rk(n)$.
\end{Theorem}

\begin{proof}[Proof of Theorem \ref{AP}]
	Without loss of generality, assume that $t_r=\mathrm{max}\{t_1, \cdots, t_r\}$. We proceed by induction on $r$ to prove the inequality \eqref{5.1}. For $r =2$, since any $2$-decomposition $(x_1, x_2)$ of $\mathrm{Div}(n)$ satisfies $x_2 = n/x_1$, $\mathcal{D}_1$ being $t_1$-chain free implies $\mathcal{D}_2$ is $t_1$-chain free, and hence $t_2$-chain free. Therefore, the inequality \eqref{5.1} is equivalent to the LYM inequality in Lemma \ref{tch} by choosing the collection $\A$ to be $\mathcal{D}_1$.
	Now we suppose that $r>2$ and that the inequality \eqref{5.1} holds for $r-1$. Then
	\begin{equation}\label{ap1}
		\begin{aligned}
			\sum_{a_{1}+\cdots+a_{r}=\rk(n)}^{} \frac{|\mathcal{D}_{a_{1},\dots,a_{r}}|}{N_{a_1,\ldots,a_r}(n)}
			&=\sum_{(x_1,\ldots, x_r)\in \mathcal{D} }^{} \frac{1}{N_{\rk(x_1), \ldots, \rk(x_r)}(n)}\\
			&=\sum_{x\in\mathcal{D}_1}\frac{1}{N_{\rk(x),\rk(n)-\rk(x)}(n)}\sum\limits_{\substack{{ (x_1,\ldots, x_r)\in \mathcal{D}}\\{x_{1} = x}}}\frac{1}{N_{\rk(x_2),\ldots,\rk(x_r)}(n/x)}.
		\end{aligned}
	\end{equation}
	It follows from the induction hypothesis that
	\[
	\sum\limits_{\substack{{ (x_1,\ldots, x_r)\in \mathcal{D}}\\{x_{1} = x}}}\frac{1}{N_{\rk(x_2),\ldots,\rk(x_r)}(n/x)} \leq t_2\cdots t_{r-1}  .
	\]
	Therefore, by Lemma \ref{tch}, the last expression in \eqref{ap1} is at most $\sigma= t_1t_2\cdots t_{r-1}$, which completes the proof of the inequality \eqref{5.1}.
	
	Note that
	\[
	\mathcal{D} = \bigsqcup_{a_{1}+\cdots+a_{r}=\rk(n)}\mathcal{D}_{a_{1},\dots,a_{r}}
	\]
	and $0\le|\mathcal{D}_{a_{1},\dots,a_{r}}|\le N_{a_1,\ldots,a_r}(n)$. Note that $|\mathcal{D}| \leq m_1 + \cdots + m_\sigma$ clearly holds when $\sigma > \binom{n+r-1}{r-1}$, assuming that $m_i = 0$ for all $i > \binom{n+r-1}{r-1}$. For the case $\sigma \leq \binom{n+r-1}{r-1}$, suppose for contradiction that $|\mathcal{D}| > m_1 + \cdots + m_\sigma$.
	Then, by Lemma \ref{d} we have
	\begin{equation*}
		\sum_{a_{1}+\cdots+a_{r}=\rk(n)}^{} \frac{|\mathcal{D}_{a_{1},\dots,a_{r}}|}{N_{a_1,\ldots,a_r}(n)} > \sigma,
	\end{equation*}
	which contradicts the inequality \eqref{5.1}. This completes the proof.
\end{proof}

An $r$-multichain in $\mathrm{Div}(n)$ is a sequenece $(x_1, \ldots, x_r)$ of $\mathrm{Div}(n)$ with $x_1 \leq x_2 \leq \cdots \leq x_r$. We denote by $\mathcal{C}(n, r)$ the collection of all $r$-multichains in $\mathrm{Div}(n)$. Let $\mathcal{C} \subseteq \mathcal{C}(n,r)$ and $a_1, \ldots, a_r$ be non-negative integers such that $a_1 + \cdots + a_r \leq \rk(n)$. Denote by 
\[
\mathcal{C}_k = \{x_k \mid (x_1, \ldots, x_k, \ldots, x_r) \in \mathcal{C}\}
\]
for each $k \in \{1, \ldots, r\}$, and denote by
\[
\mathcal{C}_{a_{1},\dots,a_r} = \{(x_1, \ldots, x_r) \in \mathcal{C}\mid \rk(x_i) = a_1 + \cdots + a_i, \; i=1, \ldots, r\}.
\]
Through a bijection similar to that in Remark \ref{remk2}, we can quickly observe that 
\[
|\mathcal{C}(n, r)_{a_1, \ldots, a_r}| = |\mathcal{D}(n, r)_{a_1, \ldots, a_r, a_{r+1}}| = N_{a_1, \ldots, a_r, a_{r+1}}(n),
\]
where $a_{r+1} = \rk(n) - \sum_{i=1}^{r}a_i$.

\begin{Theorem}\label{APc}
	Let $t_1, \ldots, t_r$ be positive integers and $\tau = t_1t_2\cdots t_r$. Suppose $\mathcal{C}$ is a family of $r$-multichains of $\mathrm{Div}(n)$ such that $\mathcal{C}_k$ is $t_k$-chain free for each $k \in \{1, \ldots, r\}$. Then
	\begin{equation}\label{5.6}
		\sum_{a_{1}+\cdots+a_{r+1}=\rk(n)}^{} \frac{|\mathcal{C}_{a_{1},\dots,a_{r}}|}{N_{a_1,\ldots,a_{r+1}}(n)} \leq \tau.
	\end{equation}
	Consequently, 
	\[
	|\mathcal{C}| \leq m_1 + \cdots + m_\tau,
	\]
	where $m_1, \ldots, m_\tau$ are the $\tau$ largest $N_{a_1,\ldots,a_{r+1}}(n)$ for non-negative integers $a_1, \ldots, a_{r+1}$ with $a_1+\cdots + a_{r+1}=\rk(n)$.
\end{Theorem}
  
\begin{proof}[Proof of Theorem \ref{APc}]
	We proceed by induction on $r$. For $r=1$, the inequality \eqref{5.6} reduces to the inequality stated in Lemma \ref{tch}. Suppose that $r>1$ and that the inequality \eqref{5.6} holds for $r-1$. Then
	\begin{equation*}
		\begin{aligned}
			&\sum_{a_1 + \cdots + a_{r+1} = \rk(n)}\frac{|\mathcal{C}_{a_{1},\dots,a_{r}}|}{N_{a_1,\ldots,a_{r+1}}(n)}\\
			&=\sum_{(x_1, \ldots, x_r)\in\mathcal{C}}\frac{1}{N_{\rk(x_1),\rk(x_2)-\rk(x_1),\ldots,\rk(x_{r})-\rk(x_{r-1}),\rk(n)-\rk(x_{r})}(n)}\\
			&=\sum_{x\in\mathcal{C}_1}\frac{1}{N_{\rk(x),\rk(n)-\rk(x)}(n)}\sum\limits_{\substack{{ (x_1,\ldots, x_r)\in \mathcal{C}}\\{x_{1} = x}}}\frac{1}{N_{\rk(x_2)-\rk(x_1),\ldots,\rk(x_{r})-\rk(x_{r-1}),\rk(n)-\rk(x_{r})}(n/x)}\\
			&\leq  \sum_{x\in\mathcal{C}_1}\frac{1}{N_{\rk(x),\rk(n)-\rk(x)}(n)} t_2\cdots t_{r}.
		\end{aligned}
	\end{equation*}
	It follows from the induction hypothesis and Lemma \ref{tch} that the last expression is at most $\tau = t_1t_2\cdots t_r$.
	
	Note that
	\[
	\mathcal{C} = \bigsqcup_{a_{1}+\cdots+a_{r+1}=\rk(n)}\mathcal{C}_{a_{1},\dots,a_{r}}
	\]
	and $0\le|\mathcal{C}_{a_{1},\dots,a_{r}}|\le N_{a_1, \ldots, a_{r+1}}(n)$ for each $\mathcal{C}_{a_{1},\dots,a_{r}}$. Then $|\mathcal{C}| \leq m_1 + \cdots + m_\tau$ clearly holds when $\tau > \binom{n+r}{r}$, assuming that $m_i = 0$ for all $i > \binom{n+r}{r}$. For the case $\tau\leq \binom{n+r}{r}$, suppose for contradiction that $|\mathcal{C}| > m_1 + \cdots + m_\tau$.
	Then, by Lemma \ref{d} we have
	\begin{equation*}
		\sum_{a_1 + \cdots + a_{r+1} = \rk(n)}\frac{|\mathcal{C}_{a_{1},\dots,a_{r}}|}{N_{a_1,\ldots,a_{r+1}}(n)} > \tau
	\end{equation*}
	which contradicts the inequality \eqref{5.6}. This completes the proof.
\end{proof}

\section*{Acknowledgements}
This paper is supported by the National Natural Science Foundation of China (Grant No. 12571350) and the Guangdong Basic and Applied Basic Research Foundation (Grant No. 2025A1515010457).

\end{document}